\documentclass{amsart}
\usepackage{mathrsfs,amssymb,amsthm,amsmath,verbatim}

\usepackage{lscape,pdflscape}

\usepackage{tikz}
\usetikzlibrary{arrows,positioning,chains,matrix,scopes}

\title{Topological games and Alster spaces}

\author[L. F. Aurichi]{Leandro F. Aurichi$^1$}
\thanks{$^1$ Partially supported by FAPESP (2010/16939-6)}
\address{Instituto de Ci\^encias Matem\'aticas e de Computa\c c\~ao,
Universidade de S\~ao Paulo, Caixa Postal 668,
S\~ao Carlos, SP, 13560-970, Brazil}
\email{aurichi@icmc.usp.br}

\author[R. R. Dias]{Rodrigo R. Dias$^2$}
\thanks{$^2$ Supported by FAPESP (2012/09214-0)}
\address{Instituto de Matem\'atica e Estat\'istica,
Universidade de S\~ao Paulo, Caixa Postal 66281,
S\~ao Paulo, SP, 05315-970, Brazil}
\email{roque@ime.usp.br}



\begin{document}

\begin{abstract}

In this paper we study connections between topological games such
as Rothberger, Menger and compact-open, and relate these games to
properties involving covers by $G_\delta$ subsets. The results
include:
(1) If Two has a winning strategy in the Menger
game on a regular space $X$, then $X$ is an Alster space. 
(2) If Two has a winning strategy in the Rothberger game on a
topological space $X$, then the $G_\delta$-topology on $X$ is
Lindel\"of.
(3) The Menger game and the compact-open game are (consistently) not
dual.

\end{abstract}

\dedicatory{Dedicated to Ofelia T. Alas on occasion of her 70th
  birthday}

\maketitle

\newtheorem{defin}{Definition}[section]

\newtheorem{prop}[defin]{Proposition}

\newtheorem{prob}[defin]{Problem}

\newtheorem{lemma}[defin]{Lemma}

\newtheorem{corol}[defin]{Corollary}

\newtheorem{example}[defin]{Example}

\newtheorem{thm}[defin]{Theorem}

\section{Topological games}

We start by recalling some definitions.
The following properties were introduced in studies of
strong measure zero and $\sigma$-compact metric spaces, respectively.

\begin{defin}[Rothberger \cite{roth}]
\label{def.roth}

A topological space $X$ is said to be a \emph{Rothberger space} if,
for every sequence $(\mathcal U_n)_{n\in\omega}$ of open covers of
$X$, there is a sequence $(U_n)_{n\in\omega}$ satisfying
$X=\bigcup_{n\in\omega}U_n$ with
$U_n\in\mathcal U_n$ for all $n\in\omega$.

\end{defin}

\begin{defin}[Hurewicz \cite{hur}]
\label{def.menger}

A topological space $X$ is said to be a \emph{Menger space} if,
for every sequence $(\mathcal U_n)_{n\in\omega}$ of open covers of
$X$, there is a sequence $(\mathcal F_n)_{n\in\omega}$ satisfying
$X=\bigcup\bigcup_{n\in\omega}\mathcal F_n$ with
$\mathcal F_n\in[\mathcal U_n]^{<\aleph_0}$ for all $n\in\omega$.

\end{defin}

The following topological games are naturally associated to the
above properties.

\begin{defin}[Galvin \cite{galvin}]
\label{def.roth-game}

The \emph{Rothberger game} in a topological space $X$ is played
according to the following rules. In each inning $n\in\omega$, 
One chooses an open cover $\mathcal U_n$ of $X$, and then  Two
chooses $U_n\in\mathcal U_n$. The play is won by Two if
$X=\bigcup_{n\in\omega}U_n$; otherwise, One is the winner.

\end{defin}

\begin{defin}[Telg\'arsky \cite{topsoe}]
\label{def.menger-game}

The \emph{Menger game} in a topological space $X$ is played as
follows. In each inning $n\in\omega$,  One chooses an open cover
$\mathcal U_n$ of $X$, and then  Two chooses a finite subset
$\mathcal F_n$ of $\mathcal U_n$. Two wins the play if
$\bigcup_{n\in\omega}\mathcal F_n$ is a cover of $X$; otherwise,
One is the winner.

\end{defin}

It is easy to see that, if One does not have a winning strategy in the
Rothberger (resp. Menger) game in a topological space $X$, then $X$ is
a Rothberger (resp. Menger) space. The following theorems show that
these properties can in fact be expressed in terms of such games.

\begin{thm}[Pawlikowski \cite{pawl}]
\label{pawl}

A topological space $X$ is Rothberger if and only if  One does
not have a winning strategy in the Rothberger game on $X$.

\end{thm}

\begin{thm}[Hurewicz \cite{hur}]
\label{menger}

A topological space $X$ is Menger if and only if  One does
not have a winning strategy in the Menger game on $X$.

\end{thm}

A more systematic study of combinatorial properties in topological
spaces was initiated by M. Scheepers in \cite{sch1}. Scheepers
introduced a framework for investigating some classes of properties
and their naturally associated games in greater generality, which has
originated the subject of \emph{selection principles}.
One such general selection principle and its associated game are
defined as follows.

\begin{defin}[Scheepers \cite{sch1}]
\label{def.S1}

Let $\mathcal{A}$ and $\mathcal{B}$ be nonempty families.
$\mathsf{S}_1(\mathcal{A},\mathcal{B})$ denotes the following
statement:
\begin{quote}
For every sequence $(A_n)_{n\in\omega}$ of elements of $\mathcal{A}$,
there is a sequence $(B_n)_{n\in\omega}$ such that
$B_n\in A_n$ for each $n\in\omega$ and
$\{B_n:n\in\omega\}\in\mathcal{B}$.
\end{quote}

\end{defin}

\begin{defin}[Scheepers \cite{sch3}]
\label{def.G1}

Let $\mathcal{A}$ and $\mathcal{B}$ be nonempty families with
$\emptyset\notin\mathcal A$. The game
$\mathsf{G}_1(\mathcal{A},\mathcal{B})$ is played as follows. In each
inning $n\in\omega$,  One chooses $A_n\in\mathcal A$, and then
 Two chooses $B_n\in A_n$. Two wins the play if
$\{B_n:n\in\omega\}\in\mathcal B$; otherwise, One is the winner.

\end{defin}

Thus the Rothberger property is the particular case $\mathsf
S_1(\mathcal O,\mathcal O)$ of Definition \ref{def.S1}, where
$\mathcal O$ denotes the family of all open covers of the space ---
more explicitly, $\mathsf S_1(\mathcal O_X,\mathcal O_X)$ means
that $X$ is a Rothberger space, where $\mathcal O_X=\{\mathcal
U\subseteq\tau_X:X=\bigcup\mathcal U\}$. It is also clear that
$X$ is a Menger space if and only if $\mathsf S_1(\mathcal
O^*_X,\mathcal O_X)$ holds, where $\mathcal
O^*_X=\{\mathcal U\in\mathcal O_X:\mathcal U$ is closed
by finite unions$\}$. Similarly, the Rothberger and Menger games can
be regarded as the games $\mathsf G_1(\mathcal O,\mathcal O)$ and
$\mathsf G_1(\mathcal O^*,\mathcal O)$ respectively.\footnote{Although
  the rules of $\mathsf G_1(\mathcal O^*,\mathcal O)$ and the Menger
  game are not quite the same, it is easy to see that these games are
  \emph{equivalent}, i.e. One (resp. Two) has a winning strategy in
  $\mathsf G_1(\mathcal O^*,\mathcal O)$ if and only if One
  (resp. Two) has a winning strategy in the Menger game.
}

The implication that was already observed in these particular cases
holds in general: namely, the nonexistence of a winning strategy for
 One in the game $\mathsf G_1(\mathcal A,\mathcal B)$ implies
$\mathsf S_1(\mathcal A,\mathcal B)$. The converse, which holds in
the particular cases previously considered (Theorems \ref{pawl} and
\ref{menger}), is not always true; see \cite[Example 3]{sch6}.

We now turn to (apparently) another game:

\begin{defin}[Galvin \cite{galvin}]
\label{def.p-o}

The \emph{point-open game} in a topological space $X$ is defined by
the following rules. In each inning $n\in\omega$,  One picks a
point $x_n\in X$, and then  Two chooses an open set
$U_n\subseteq X$ with $x_n\in U_n$. The play is won by One if
$X=\bigcup_{n\in\omega}U_n$; otherwise, Two is the winner.

\end{defin}

In \cite{galvin}, F. Galvin showed that the point-open game is
essentially the same as the Rothberger game, in the following sense.

We say that two games $\mathsf G$ and $\mathsf G'$ are \emph{dual} if
\begin{itemize}
\item[$\cdot$] One has a winning strategy in $\mathsf G$ if and only
  if Two has a winning strategy in $\mathsf G'$; and
\item[$\cdot$] Two has a winning strategy in $\mathsf G$ if and only
  if One has a winning strategy in $\mathsf G'$.
\end{itemize}

\begin{thm}[Galvin \cite{galvin}]
\label{dual-galvin}

The Rothberger game and the point-open game are dual.

\end{thm}

We may then wonder how the point-open game could be modified to
produce a similar game that is dual to the Menger game.
The following is a natural candidate:

\begin{defin}[Telg\'arsky \cite{telg1}]
\label{def.c-o}

The \emph{compact-open game} in a topological space $X$ is defined as
follows: in each inning $n\in\omega$,  One chooses a compact
subset $K_n$ of $X$, and then  Two chooses an open subset $U_n$
of $X$ such that $K_n\subseteq U_n$; the play is won by One if
$X=\bigcup_{n\in\omega}U_n$, otherwise Two is the winner.

\end{defin}

In \cite[Corollary 3]{topsoe}, R. Telg\'arsky proved that One has a
winning strategy in the compact-open game if and only if Two has a
winning strategy in the Menger game. Telg\'arsky also observes
(in Proposition 1 of the same paper) that One having a winning
strategy in the Menger game implies Two having a winning strategy in
the compact-open game, and then asks:

\begin{prob}[Telg\'arsky \cite{topsoe}]
\label{dual?}

Does the converse hold? I.e., are the Menger game and the compact-open
game dual?

\end{prob}

As we shall see later, in Examples \ref{cov<d} and \ref{sierp+sorg},
this may not always be the case.

The relationship between these two games may be more clearly
understood by considering the following.

\begin{defin}
\label{def.Ccov}

Let $X$ be a topological space. An open cover
$\mathcal{U}$ of $X$ is said to be a \emph{$k$-cover}
of $X$ if for every compact subset $K$ of $X$ there is
$U\in\mathcal{U}$ such that $K\subseteq U$. The family of all
$k$-covers of $X$ will be denoted by $\mathcal{K}_X$.

\end{defin}

%
%
%
%
%
%
%

The next result is a particular case of Theorem 6.2 of
\cite{telg2}.

\begin{prop}[Galvin, Telg\'arsky \cite{telg2}]
\label{dualc-o}

The game $\mathsf{G}_1(\mathcal{K},\mathcal{O})$ and the compact-open
game are dual.

\end{prop}

Problem \ref{dual?} may then be rewritten as:

\begin{prob}[Telg\'arsky \cite{topsoe}]
\label{dual2?}

Does the existence of a winning strategy for One in $\mathsf
G_1(\mathcal K,\mathcal O)$ imply the existence of a winning strategy
for One in the game $\mathsf G_1(\mathcal O^*,\mathcal O)$?

\end{prob}

Note that $\mathcal O^*\subseteq\mathcal K$; thus, a counterexample to
Problem \ref{dual?} (i.e. Problem \ref{dual2?}) must be a space in
which these two classes of open covers are, in a certain sense, very
far from each other.
These problems shall be further discussed in Section \ref{section.ex}.

As has already been observed, if One does not have a winning strategy
in $\mathsf{G}_1(\mathcal{K},\mathcal{O})$ (i.e. Two does not have
a winning strategy in the compact-open game), then
$\mathsf{S}_1(\mathcal{K},\mathcal{O})$ holds.
The question of whether the converse holds remains unsettled:

\begin{prob}
\label{c-o-equiv?}

Is $\mathsf{S}_1(\mathcal{K},\mathcal{O})$ equivalent to One not
having a winning strategy in the game
$\mathsf{G}_1(\mathcal{K},\mathcal{O})$ (i.e. Two not
having a winning strategy in the compact-open game)?

\end{prob}

%
%
%
%
%
%
%
%
%
%
%
%
%

\section{Alster spaces}

We now turn to properties of covers of topological spaces by
$G_\delta$ subsets. The main object of our interest is the
\emph{Alster property}, defined below (Definition \ref{def.alster}).

\begin{defin}
\label{alstercover}

Let $X$ be a topological space. A cover $\mathcal{W}$ of $X$ by
$G_\delta$ subsets is said to be an \emph{Alster cover} if every
compact subset of $X$ is included in some element of $\mathcal W$.
The set of all Alster covers of $X$ will be denoted by
$\mathscr{A}_X$.

\end{defin}

\begin{defin}[Alster \cite{alster}]
\label{def.alster}

A topological space $X$ is an \emph{Alster space}
if, for every
$\mathcal U\in\mathscr A_X$, there is a countable $\mathcal
V\subseteq\mathcal U$ with $X=\bigcup\mathcal V$.

\end{defin}

Alster spaces were introduced in \cite{alster} in an attempt to
characterize the class of \emph{productively Lindel\"of spaces},
i.e. the class of topological spaces $X$ such that $X\times Y$ is
Lindel\"of whenever $Y$ is a Lindel\"of space.

\begin{thm}[Alster \cite{alster}]
\label{thm.alster}

Alster spaces are productively Lindel\"of. Assuming the Continuum
Hypothesis, productively Lindel\"of spaces of weight not exceeding
$\aleph_1$ are Alster.

\end{thm}

An internal characterization of productive Lindel\"ofness --- a
problem attributed to H. Tamano in \cite[Problem 5]{przy} ---
is still unknown. The following problem, implicitly raised in
\cite{alster} (see also \cite{barr2}), remains open:

\begin{prob}[Alster \cite{alster}, Barr-Kennison-Raphael \cite{barr2}]
\label{alster?}

Is is true that every productively Lindel\"of space is an Alster
space?

\end{prob}

After observing that both the properties ``$X$ is an Alster space''
and ``Two has a winning strategy in the Menger game on $X$'' are
implied by ``$X$ is $\sigma$-compact''\footnote{A space is
  \emph{$\sigma$-compact} if it is a countable union of compact
  subsets.} and imply ``$X$ is Lindel\"of and every continuous image
of $X$ in a separable metrizable space is $\sigma$-compact'',
F. Tall asks in \cite[Problem 5]{tallD}:

\begin{prob}[Tall \cite{tallD}]
\label{tall?}

Is there any implication between the Alster property and Two having a
winning strategy in the Menger game?

\end{prob}

We shall provide a complete answer to Problem \ref{tall?} by showing
that:
\begin{itemize}

\item[$\cdot$]
if a Two has a winning strategy in the Menger game on a regular space
$X$, then $X$ is Alster (Corollary \ref{menger->alster});

\item[$\cdot$]
the regularity hypothesis in the above result is essential (Example
\ref{R-ctbl});

\item[$\cdot$]
the converse implication does not hold (Example \ref{Ytelg}).

\end{itemize}

In what follows, we will denote by $\mathcal O^\delta_X$ the family
of all covers of a topological space $X$ by $G_\delta$ subsets.
We start by proving a characterization of the Alster property in terms
of the selection principle $\mathsf S_1$.\footnote{Proposition
  \ref{alster-sel} has also been obtained (independently) by
  L. Babinkostova, B. Pansera and M. Scheepers in \cite{bab}.}

\begin{prop}
\label{alster-sel}

A topological space $X$ is an Alster space if and only if
$\mathsf{S}_1(\mathscr A_X,\mathcal O^\delta_X)$ holds.

\end{prop}

\begin{proof}

The converse is clear.
For the direct implication,
suppose that $X$ is an Alster space and let
$(\mathcal{U}_n)_{n\in\omega}$ be a sequence in $\mathscr A_X$.
Let $S=\prod_{n\in\omega}\mathcal{U}_n$ and, for each $f\in
S$, define $V_f=\bigcap_{n\in\omega}f(n)$. It follows that
$\{V_f:f\in S\}$ is an Alster cover of $X$; therefore,
there is $\{f_n:n\in\omega\}\subseteq S$ such
that $X=\bigcup_{n\in\omega}V_{f_n}$. Now,
for every $n\in\omega$, define $A_n=f_n(n)$.
Then
$$
X=\bigcup_{n\in\omega}V_{f_n}=\bigcup_{n\in\omega}\bigcap_{k\in\omega}f_n(k)\subseteq\bigcup_{n\in\omega}f_n(n)=\bigcup_{n\in\omega}A_n
$$
and, since $A_n
\in\mathcal{U}_n$ for all $n\in\omega$, we are done.
\end{proof}

\begin{corol}
\label{alster->S1(K,O)}

Every Alster space satisfies $\mathsf S_1(\mathcal K,\mathcal O)$.

\end{corol}

\begin{proof}

This is immediate in view of Proposition \ref{alster-sel}, since
$\mathcal K\subseteq\mathscr A$.
\end{proof}

Let us now consider a natural modification of the compact-open game.

\begin{defin}[Telg\'arsky \cite{telg2}]
\label{def.c-Gd}

The \emph{compact-$G_\delta$ game} in a topological space $X$ is
defined in the same way as the compact-open game, with the difference
that now Two is allowed to play $G_\delta$ subsets of $X$.

\end{defin}

The proof of the following result is analogous to the proof of
Proposition \ref{dualc-o} --- see \cite[Theorem 6.2]{telg2}.

\begin{prop}
\label{c-gd}
The game $\mathsf{G}_1(\mathscr A,\mathcal{O}^\delta)$ and the
compact-$G_\delta$ game are dual.
\end{prop}

Propositions \ref{alster-sel} and \ref{c-gd} yield:

\begin{corol}
\label{alstergame}

If Two does not have a winning strategy in the compact-$G_\delta$ game
on a topological space $X$, then $X$ is an Alster space.

\end{corol}

As usual, the characterization of the selective property in terms of
its naturally associated game is of interest:

\begin{prob}
\label{alst-equiv}
Is the Alster property equivalent to  Two not having a winning
strategy in the compact-$G_\delta$ game?
\end{prob}

%
%
%

Finally, Theorem 5.1 of \cite{telg2} and Corollary 3 of \cite{topsoe}
can be combined to yield:

\begin{thm}[Telg\'arsky \cite{telg2,topsoe}]
\label{telg-topsoe}

Consider the following statements about a topological space $X$:

\begin{itemize}

\item[$(a)$]
One has a winning strategy in the compact-$G_\delta$ game on $X$;

\item[$(b)$]
One has a winning strategy in the compact-open game on $X$;

\item[$(c)$]
Two has a winning strategy in the Menger game on $X$.

\end{itemize}

Then $(a)\leftrightarrow(b)\rightarrow(c)$. Furthermore, if $X$ is
regular, then the three statements are equivalent.

\end{thm}

This allows us to relate the Menger game to the Alster property:

\begin{corol}
\label{menger->alster}

If  Two has a winning strategy in the Menger game on a
regular space $X$, then $X$ is an Alster space.

\end{corol}

\begin{proof}

By Theorem \ref{telg-topsoe} and Corollary \ref{alstergame}.
\end{proof}

In \cite[Corollary 4]{topsoe}, Telg\'arsky showed:

\begin{thm}[Telg\'arsky \cite{topsoe}]
\label{telg.metr}

If $X$ is a metrizable space, then Two has a winning strategy in the
Menger game on $X$ if and only if $X$ is $\sigma$-compact.

\end{thm}

In \cite{bz}, T. Banakh and L. Zdomskyy noted that Telg\'arsky's
argument would follow with ``regular hereditarily Lindel\"of'' in
place of ``metrizable''.
Since hereditarily Lindel\"of regular spaces have the property
that every compact subset is a $G_\delta$ (a condition that clearly
implies $\sigma$-compactness in the presence of the Alster property),
Corollary \ref{menger->alster} extends the Banakh-Zdomskyy version of
Theorem \ref{telg.metr}.

Telg\'arsky's proof of the equivalence between $(b)$ and $(c)$ in
Theorem \ref{telg-topsoe} is rather indirect.
Inspired by \cite{direct}, we can give a more straightforward proof of
Corollary \ref{menger->alster}, which does not depend on the
aforementioned equivalence:

\begin{proof}[An alternative proof of Corollary \ref{menger->alster}]

Let $\sigma:\mbox{}^{<\omega}\mathcal
O_X\setminus\{\emptyset\}\rightarrow[\tau_X]^{<\aleph_0}$ be a winning
strategy for Two in the Menger game on $X$. Now let $\mathcal W$ be an
Alster cover of $X$. Our task is to find a countable subset of
$\mathcal W$ that covers $X$.

The following claim is taken from \cite{direct}.

\emph{Claim 1.}
For every $s\in\mbox{}^{<\omega}\mathcal O_X$, the set
$K_s=\bigcap\{\overline{\bigcup\sigma(s^\frown\mathcal
U)}:\mathcal U\in\mathcal O_X\}$ is compact.

Indeed, let $\mathcal V$ be a cover of $K_s$ by open subsets of $X$.
For each $x\in K_s$, pick an open neighbourhood $U_x$ of $X$ such that
$\overline{U_x}$ is included in some element of $\mathcal V$; now, for
each $x\in X\setminus K_s$, pick an open neighbourhood $U_x$ of $X$
with $\overline{U_x}\subseteq X\setminus K_s$. Consider then $\mathcal
U_0=\{U_x:x\in X\}\in\mathcal O_X$, and let $F\in[X]^{<\aleph_0}$ be
such that $\sigma(s^\frown\mathcal U_0)=\{U_x:x\in F\}$. Note that
$K_s\subseteq\overline{\bigcup\sigma(s^\frown\mathcal
U_0)}=\bigcup\{\overline{U_x}:x\in F\}$; thus, if for each $x\in
F\cap K_s$ we pick a $V_x\in\mathcal V$ with $\overline{U_x}\subseteq
V_x$, we will have $K_s\subseteq\bigcup\{V_x:x\in F\cap K_s\}$.
This proves Claim 1.

For each $s\in\mbox{}^{<\omega}\mathcal O_X$,
we can then fix a $W_s\in\mathcal W$ with $K_s\subseteq W_s$.

\emph{Claim 2.}
For every $s\in\mbox{}^{<\omega}\mathcal O_X$, there is a countable
$\mathcal C_s\subseteq\mathcal O_X$ such that
$K_s\subseteq\bigcap\{\overline{\bigcup\sigma(s^\frown\mathcal
  U)}:\mathcal U\in\mathcal C_s\}\subseteq W_s$.

Since $K_s\subseteq W_s$, the set
$\{X\setminus\overline{\bigcup\sigma(s^\frown\mathcal U)}:\mathcal
U\in\mathcal O_X\}$ is an open cover of $X\setminus W_s$. But
$X\setminus W_s$ is an $F_\sigma$-subset of $X$; since our hypothesis
implies that $X$ is a Lindel\"of space, it follows that $X\setminus
W_s$ is Lindel\"of as well, whence there is a countable $\mathcal
C_s\subseteq\mathcal O_X$ such that $X\setminus
W_s\subseteq\bigcup\{X\setminus\overline{\bigcup\sigma(s^\frown\mathcal
  U)}:\mathcal U\in\mathcal C_s\}$. This proves Claim 2.

Now define recursively $\mathcal A_0=\mathcal C_\emptyset$ and
$\mathcal A_{n+1}=\mathcal A_n\cup\bigcup\{\mathcal
C_s:s\in\mbox{}^{n+1}\mathcal A_n\}$ for all
$n\in\omega$.
Let $\mathcal A=\bigcup\{\mathcal A_n:n\in\omega\}$.
We will show that the countable subset
$\mathcal W_0=\{W_s:s\in\mbox{}^{<\omega}\mathcal A\}$ of $\mathcal W$
is a cover of $X$.

Suppose, to the contrary, that there is $p\in
X\setminus\bigcup\mathcal W_0$.
Since $p\notin W_\emptyset$, there is some $\mathcal U_0\in\mathcal
C_\emptyset$ such that $p\notin\overline{\bigcup\sigma((\mathcal
  U_0))}$. We also have $p\notin W_{(\mathcal U_0)}$, so there is
$\mathcal U_1\in\mathcal C_{(\mathcal U_0)}$ such that
$p\notin\overline{\bigcup\sigma((\mathcal U_0,\mathcal U_1))}$.
By proceeding in this fashion ($p\notin W_{(\mathcal U_0,\mathcal
  U_1)}$, and so on), we obtain a play
$$
(\mathcal U_0,\sigma((\mathcal U_0)),\mathcal U_1,\sigma((\mathcal
U_0,\mathcal U_1)),\mathcal U_2,\sigma((\mathcal U_0,\mathcal
U_1,\mathcal U_2)),\mathcal U_3,\dots)
$$
of the Menger game on $X$ such that
$p\notin\overline{\bigcup\sigma((\mathcal U_0,\mathcal
  U_1,\dots,\mathcal U_k))}$ for all $k\in\omega$. But this is a
contradiction, since  Two follows the winning strategy $\sigma$
in this play.
\end{proof}

A similar argument shows that,
if ``Menger'' is replaced by ``Rothberger'' in
Proposition \ref{menger->alster},
the conclusion can be replaced by ``$X_\delta$ is Lindel\"of'' ---
here, $X_\delta$ is the set $X$ endowed with the topology generated by
the $G_\delta$ subsets from its original topology.
But in this case we can avoid the requirement of any separation axioms
by making use of Theorem \ref{dual-galvin} (cf. Theorem 6.1 of
\cite{telg1} and Theorem 2 of \cite{galvin}):

\begin{prop}
\label{roth->delta}

If Two has a winning strategy in the Rothberger game on a
topological space $X$, then $X_\delta$ is a Lindel\"of space.

\end{prop}

\begin{proof}

By Theorem \ref{dual-galvin}, this hypothesis is equivalent to the
existence of a winning strategy for  One in the point-open game.
Let then $\sigma:\mbox{}^{<\omega}\tau\rightarrow X$ be such a
strategy, where $\tau$ is the topology of $X$.

Now let $\mathcal W$ be a cover of $X$ by $G_\delta$ subsets. For each
$W\in\mathcal W$, fix a sequence $(U(W,n))_{n\in\omega}$ of open sets
with $W=\bigcap_{n\in\omega}U(W,n)$. Proceeding by induction on
$n\in\omega$, we shall assign to each $s\in\mbox{}^n\omega$ an element
$W_s$ of $\mathcal W$ as follows.

First, pick $W_\emptyset\in\mathcal W$ such that $\sigma(\emptyset)\in
W_\emptyset$. Now let $n\in\omega$ be such that
$W_s\in\mathcal W$ has already been defined for all
$s\in\mbox{}^n\omega$. For each $s\in\mbox{}^n\omega$ and each
$k\in\omega$, choose $W_{s^\smallfrown k}\in\mathcal W$ satisfying
$\sigma(t_{s,k})\in W_{s^\smallfrown k}$, where
$t_{s,k}\in\mbox{}^{n+1}\tau$ is the sequence defined by
$t_{s,k}(i)=U(W_{s\upharpoonright i},s(i))$ for all $i<n$ and
$t_{s,k}(n)=U(W_s,k)$.

We claim that $\{W_s:s\in\mbox{}^{<\omega}\omega\}\subseteq\mathcal W$
is a cover of $X$. Suppose not, and fix $p\in
X\setminus\bigcup_{s\in\mbox{}^{<\omega}\omega}W_s$.
For $n\in\omega$, we can recursively pick $k_n\in\omega$ with
$p\notin U(W_{(k_i)_{i<n}},k_n)$.
But then we get a contradiction from the fact that
$$
(\sigma(\emptyset),U(W_\emptyset,k_0),\sigma(U(W_\emptyset,k_0)),U(W_{(k_0)},k_1),\sigma(U(W_{(k_0)},k_1)),U(W_{(k_0,k_1)},k_2),\dots)
$$
is a play of the point-open game in which One plays according to
$\sigma$ and loses.
\end{proof}

We shall see later, in Example \ref{R-ctbl}, that the regularity
hypothesis in Corollary \ref{menger->alster} is essential.

The following diagram summarizes the connections between the
properties considered in this paper. We will now quote some results
from which some of the implications in the diagram follow.

%
%

For the first result (proven in Theorem 5.2 of \cite{rice}),
recall that a topological space $X$ is \emph{scattered} if every
nonempty subspace $Y\subseteq X$ has an isolated point (relative to
$Y$).

\begin{thm}[Levy-Rice \cite{rice}]
\label{Xdelta}

If a regular space $X$ is Lindel\"of and scattered, then $X_\delta$ is
Lindel\"of.

\end{thm}

The next result is attributed to F. Galvin in \cite{gerlits} ---
see Theorem 47 of \cite{schtall}. Recall that a \emph{$P$-space} is a
topological space in which every $G_\delta$ subset is open.

\begin{prop}[Galvin]
\label{P-roth}

A $P$-space is Lindel\"of if and only if it is Rothberger.

\end{prop}

Recall that a \emph{Michael space} is a Lindel\"of space $X$
such that $X\times\omega^\omega$ is not Lindel\"of. Michael spaces
have been constructed with the aid of several set-theoretical
hypotheses; see e.g. \cite{michael}, \cite{lawrence} and \cite{moore}.
In Proposition 3.1 of \cite{repovs}, it was shown:\footnote{We thank
  Lyubomyr Zdomskyy for bringing this result to our attention.}

\begin{thm}[Repov\v s-Zdomskyy \cite{repovs}]
\label{michael}

If there is a Michael space,
then every productively Lindel\"of space is Menger.

\end{thm}

Finally, a closer look at the proof of the folklore fact that
compact scattered spaces are Rothberger (see e.g.
\cite[Proposition 5.5]{aurichi}) shows that, in fact, we have:

\begin{prop}[folklore]
\label{comp-scattered}

Two has a winning strategy in the Rothberger game on every compact
scattered space.

\end{prop}

Each arrow of the diagram has the number of the result
from which the implication follows, as well as the number of the
counterexample (in Section \ref{section.ex}) showing that
the implication cannot be reversed --- or that the
regularity assumption is necessary, if this is the case.


\newpage

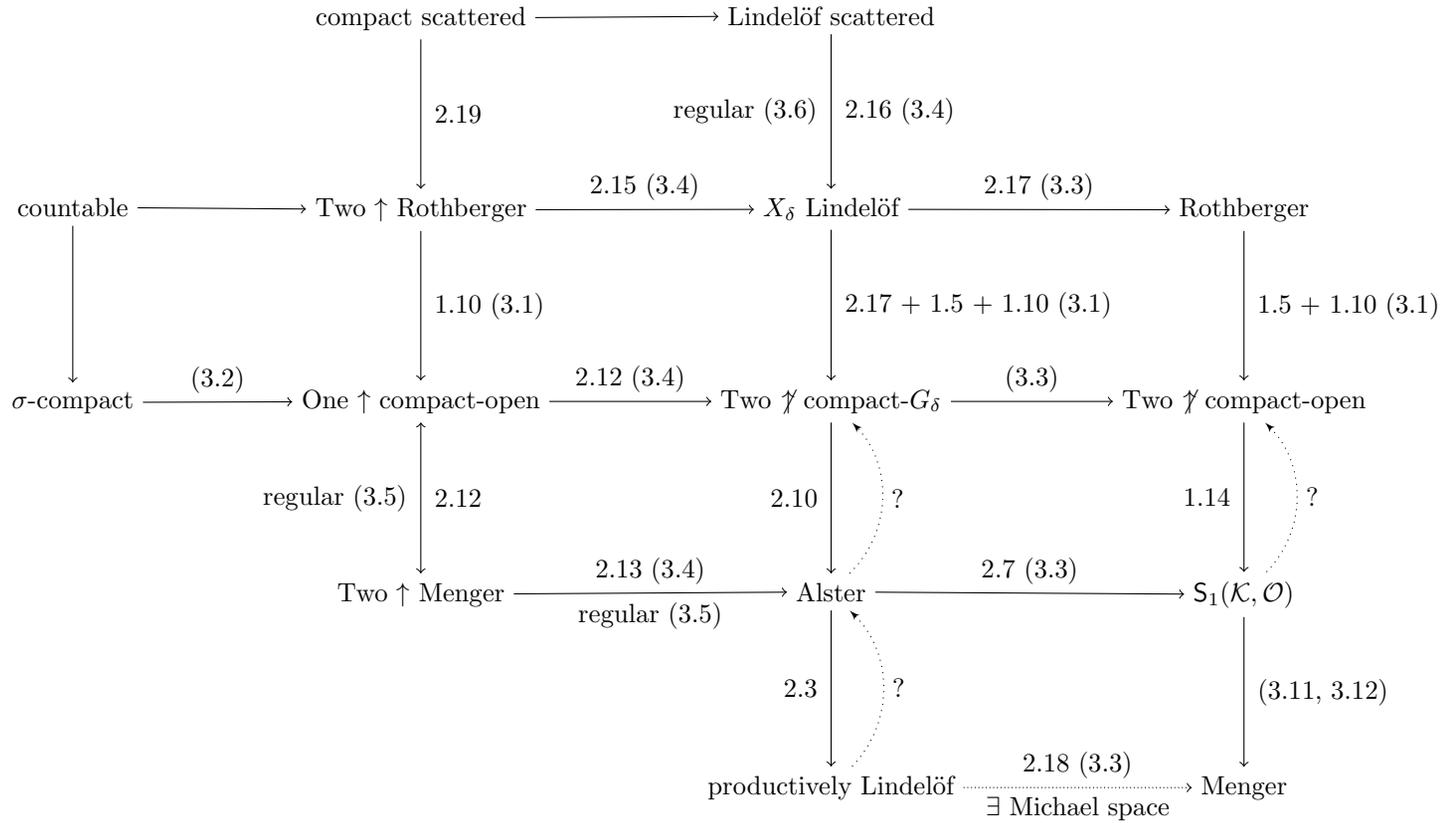
\begin{landscape}

\begin{tikzpicture}

\matrix (m) [matrix of nodes, row sep=20mm, column sep=20mm,
]
{
& compact scattered & Lindel\"of scattered & \\
countable & Two $\uparrow$ Rothberger &
$X_\delta$ Lindel\"of & Rothberger \\
$\sigma$-compact & One $\uparrow$ compact-open &
Two $\not\mbox{}\!\!\!\uparrow$ compact-$G_\delta$ &
Two $\not\mbox{}\!\!\!\uparrow$ compact-open \\
& Two  $\uparrow$ Menger & Alster & $\mathsf
S_1(\mathcal K,\mathcal O)$ \\
& & productively Lindel\"of & Menger \\
};

\draw[->] (m-1-2) to (m-1-3);
\draw[->] (m-2-1) to (m-2-2);
\draw[->] (m-2-2) to node[auto]{\ref{roth->delta} (\ref{Ytelg})} (m-2-3);
\draw[->] (m-2-3) to node[auto]{\ref{P-roth} (\ref{Lspace})} (m-2-4);
\draw[->] (m-3-1) to node[auto]{(\ref{fortissimo})} (m-3-2);
\draw[->] (m-3-2) to node[auto]{\ref{telg-topsoe} (\ref{Ytelg})} (m-3-3);
\draw[->] (m-3-3) to node[auto]{(\ref{Lspace})} (m-3-4);
\draw[->] (m-4-2) to node[auto]{\ref{menger->alster} (\ref{Ytelg})}
node[auto,swap]{regular (\ref{R-ctbl})} (m-4-3);
\draw[->] (m-4-3) to node[auto]{\ref{alster->S1(K,O)} (\ref{Lspace})} (m-4-4);
\draw[->,densely dotted] (m-5-3) to node[auto,swap]{$\exists$ Michael space} node[auto]{\ref{michael} (\ref{Lspace})} (m-5-4);

\draw[->] (m-2-1) to (m-3-1);
\draw[->] (m-2-2) to node[auto]{\,\ref{dual-galvin} (\ref{2^omega})} (m-3-2);
\draw[<->] (m-3-2) to node[auto,swap]{regular (\ref{R-ctbl})\,} node[auto]{\,\ref{telg-topsoe}} (m-4-2);
\draw[->] (m-1-2) to node[auto]{\,\ref{comp-scattered}} (m-2-2);
\draw[->] (m-1-3) to node[auto]{\,\ref{Xdelta} (\ref{Ytelg})}
node[auto,swap]{regular (\ref{luzin-michael})\,} (m-2-3);
\draw[->] (m-2-3) to node[auto]{\,\ref{P-roth} + \ref{pawl} + \ref{dual-galvin} (\ref{2^omega})} (m-3-3);
\draw[->] (m-3-3) to node[auto,swap]{\ref{alstergame}\,} (m-4-3);
\draw[->] (m-4-3) to node[auto,swap]{\ref{thm.alster}\,} (m-5-3);
\draw[->] (m-2-4) to node[auto]{\,\ref{pawl} + \ref{dual-galvin} (\ref{2^omega})} (m-3-4);
\draw[->] (m-3-4) to node[auto,swap]{\ref{dualc-o}\,} (m-4-4);
\draw[->] (m-4-4) to node[auto]{\,(\ref{cov<d}, \ref{sierp+sorg})} (m-5-4);

\draw[->,>=latex',dotted,bend right=45] (m-4-3) to node[auto,swap]{?} (m-3-3);
\draw[->,>=latex',dotted,bend right=45] (m-4-4) to node[auto,swap]{?} (m-3-4);
\draw[->,>=latex',dotted,bend right=45] (m-5-3) to node[auto,swap]{?} (m-4-3);


\end{tikzpicture}

\end{landscape}

\newpage

\section{Counterexamples}

\label{section.ex}

We shall now see that, unless the question of whether the converse
implication holds is indicated, the implications from the previous
diagram cannot be reversed (at least consistently); moreover, the
regularity assumptions that appear in the diagram cannot be dropped
(again, at least consistently). This will follow from the examples
listed below.

\begin{example}
\label{2^omega}

A compact Hausdorff space that is not Rothberger.

\end{example}

The Cantor set $2^\omega$ satisfies these conditions: it is folklore
that the sequence $(\mathcal U_n)_{n\in\omega}$ of open covers of
$2^\omega$ defined by $\mathcal
U_n=\{\pi_n^{-1}[\{0\}],\pi_n^{-1}[\{1\}]\}$, where
$\pi_n:2^\omega\rightarrow 2$ is the projection onto the $n$-th
coordinate, witnesses the failure of the Rothberger property.

\begin{example}
\label{fortissimo}

A Lindel\"of scattered regular space that is not $\sigma$-compact and
such that Two has a winning strategy in the Rothberger game.

\end{example}

This is the one-point Lindel\"ofication of an uncountable discrete
space, i.e. the space $X=A\cup\{p\}$, where $A$ is
uncountable and $p\notin A$, in which every point of $A$ is isolated
and cocountable subsets of $X$ are open. Two has a winning strategy in
the Rothberger game since, if she covers the point $p$ in the first
inning, only countably many points remain uncovered.

\begin{example}
\label{Lspace}

A Rothberger regular space that is not productively Lindel\"of.

\end{example}

J. Moore's L space \cite{jmooreL} is Rothberger
(see \cite[section 4]{schtall})
but has a non-Lindel\"of finite power\footnote{We thank Marion
  Scheepers and Boaz Tsaban for pointing this out to us.}
(see \cite[Theorem 3.4(2)]{tz} and \cite[Theorem 2]{arh}).

\begin{example}
\label{Ytelg}

There is a Lindel\"of regular non-scattered space $Y$ such that
$Y_\delta$ is Lindel\"of (hence, in particular, $Y$ is Alster) and Two
does not have a winning strategy in the Menger game on $Y$.

\end{example}

Let $Y$ be the space considered by Telg\'arsky in Section 7 of
\cite{telg2}: for each
$\lambda\in\lim(\omega_1)=\{\gamma\in\omega_1:\gamma$ is a limit
ordinal$\}$, fix a cofinal subset $C_\lambda\subseteq\lambda$ such
that $|C_\lambda\cap\alpha|<\aleph_0$ whenever $\alpha<\lambda$; the
set
$Y=\{\chi_{C_\lambda}:\lambda\in\lim(\omega_1)\}\cup\{\chi_F:F\in[\omega_1]^{<\aleph_0}\}$
is then regarded as a subspace of $2^{\omega_1}$ with the countable
box product topology $\tau$.\footnote{Here, $\chi_A$ denotes the
  function in $\mbox{}^{\omega_1}2$ satisfying
  $\{\alpha\in\omega_1:\chi_A(\alpha)=1\}=A$.}

Note that $Y=Y_\delta$ and that
$\{\chi_F:F\in[\omega_1]^{<\aleph_0}\}$ is a closed subset of $Y$
without isolated points.
In \cite{telg2}, it was proven that the compact-open game on $Y$ is
undetermined; thus, in view of Theorem \ref{telg-topsoe}, there is no
winning strategy for Two in the Menger game on $Y$.

As in Corollary \ref{menger->alster}, we can give a proof for this
last fact that does not rely on the equivalence between $(b)$ and
$(c)$ of Theorem \ref{telg-topsoe}.

\begin{proof}[A direct proof for Example \ref{Ytelg}]

For each $p\in Y$ and each
$\alpha\in\omega_1$, we shall write $V(p,\alpha)=\{y\in
Y:y\!\upharpoonright\!\alpha=p\!\upharpoonright\!\alpha\}\in\tau$.

Let $\sigma$ be a strategy for Two in the Menger game on $Y$.
By expanding the answers of  Two if necessary, we may
regard $\sigma$ as a function
$\sigma:(\mbox{}^{<\omega}\lim(\omega_1))\setminus\{\emptyset\}\rightarrow[\omega_1]^{<\aleph_0}$,
meaning that, if One gives an open cover $\{V(y,\alpha):y\in Y\}$ of
$Y$ with $\alpha\in\lim(\omega_1)$ --- note that, as $Y$ is Lindel\"of
and $\omega_1$ is regular, any open cover of $Y$ has an open
refinement of this form ---, Two responds by choosing, for some
$F\in[\alpha]^{<\aleph_0}$, the open sets $V(y,\alpha)$ with
$y\in\{\chi_G:G\subseteq F\}\cup\{\chi_{C_\gamma}:\gamma\in
F\cap\lim(\omega_1)\}\cup\{\chi_{C_\alpha}\}$.

For each $t\in\mbox{}^{<\omega}\lim(\omega_1)$, we have
$\max(\sigma(t^\frown\alpha))<\alpha$ for all
$\alpha\in\lim(\omega_1)$; thus, it follows from Fodor's Lemma
(\cite[Theorem 2]{fodor}; see e.g. \cite[Theorem 21.12]{jw2}) that
there exist $\beta_t\in\omega_1$ and a stationary set
$S_t\subseteq\lim(\omega_1)$ such that, for all $\alpha\in S_t$, we
have $\max(\sigma(t^\frown\alpha))=\beta_t$. Let $M$ be a countable
elementary submodel of $H_\theta$ for a convenient choice of $\theta$
(see e.g. \cite{dow} or \cite[Chapter 24]{jw2}) such that
$Y,\tau,\sigma\in M$, and consider
$\lambda=M\cap\omega_1\in\lim(\omega_1)$. By elementarity, it follows
that, for each $t\in\mbox{}^{<\omega}\lim(\lambda)$, there exist
$\beta_t\in\lambda$ and an unbounded subset $S_t$ of $\lambda$ with
$S_t\subseteq\lim(\lambda)$ such that
$\max(\sigma(t^\frown\alpha))=\beta_t$ for all $\alpha\in S_t$.

We shall now prove that $\sigma$ is a not a winning strategy by
showing that One can prevent the point $\chi_{C_\lambda}\in Y$ from
being covered if Two plays according to $\sigma$.

In order to accomplish this, One starts by picking $\xi_0\in
C_\lambda$ with $\beta_\emptyset<\xi_0$, and then plays $\alpha_0\in
S_\emptyset$ such that $\xi_0<\alpha_0$ --- which, we recall, is short
for saying that she plays the open cover $\{V(y,\alpha):y\in
Y\}$. Since Two follows the strategy $\sigma$, he responds with
$\sigma((\alpha_0))\in[\lambda]^{<\aleph_0}$; now One picks $\xi_1\in
C_\lambda$ satisfying $\beta_{(\alpha_0)}<\xi_1$, and then plays
$\alpha_1\in S_{(\alpha_0)}$ with $\xi_1<\alpha_1$. In general, in the
$n$-th inning, if $t_n=(\alpha_k)_{k<n}$ is the sequence of One's
moves so far, she picks $\xi_n\in C_\lambda$ such that
$\beta_{t_n}<\xi_n$, and then plays $\alpha_n\in S_{t_n}$ with
$\xi_n<\alpha_n$. It is clear that the point $\chi_{C_\lambda}\in Y$
is not covered in any of the innings, since for all $n\in\omega$ we
have $\max\sigma((\alpha_0,\dots,\alpha_n))<\xi_n<\alpha_n$ and
$\chi_{C_\lambda}(\xi_n)=1$.
\end{proof}

\begin{example}
\label{R-ctbl}

There is a Hausdorff non-regular space $X$ such that $\mathsf
S_1(\mathcal K_X,\mathcal O_X)$ fails and yet Two has a winning
strategy in the Menger game on $X$ --- in particular, $X$ is a
Menger space.

\end{example}

This is the space $X$ obtained by taking the real line $\mathbb R$
(with the usual topology) and then declaring every countable subset
closed. Since every compact subset of $X$ is finite, it follows
from Theorem 17 of \cite{sch1} that $\mathsf S_1(\mathcal K_X,\mathcal
O_X)$ is equivalent to $S_1(\mathcal O_X,\mathcal O_X)$, which does
not hold since $\mathbb R$ is not a Rothberger space.

Now write $\{2k+1:k\in\omega\}=\dot\bigcup_{j\in\omega}A_j$ with
$|A_j|=\aleph_0$ for each $j\in\omega$, and let $\varrho$ be a winning
strategy for  Two in the Menger game played on the real line
with the usual topology --- such a strategy exists since $\mathbb R$
is $\sigma$-compact. We may assume that, in the Menger game on $X$,
One only plays covers constituted by basic open sets of the form
$U\setminus C$, where $U$ is open in $\mathbb R$ and
$C\subseteq\mathbb R$ is countable; for each such basic open set $W$,
fix $U(W)$ open in $\mathbb R$ and $C(W)\in[\mathbb R]^{\le\aleph_0}$
with $W=U(W)\setminus C(W)$; then, for each basic open cover $\mathcal
W$ of $X$, define $\mathcal U(\mathcal W)=\{U(W):W\in\mathcal
W\}\in\mathcal O_\mathbb R$.

We shall now describe a winning strategy for  Two in the Menger
game on $X$. In each even inning $2k\in\omega$, if $(\mathcal
W_i)_{i\le 2k}$ is the sequence of open covers played by One so far,
Two responds with $\mathcal F_{2k}\in[\mathcal W_{2k}]^{<\aleph_0}$
such that
$\varrho((\mathcal U(\mathcal W_{2i}))_{i\le k})=\{U(W):W\in\mathcal
F_{2k}\}$
--- i.e., $\{U(W):W\in\mathcal F_{2k}\}$ is Two's answer to the
sequence $(\mathcal U(\mathcal W_{2i}))_{i\le k}$ in the Menger game
on $\mathbb R$ according to the strategy $\varrho$. Now Two makes use
of the innings in $A_k$ to cover the countably many points in
$\bigcup_{W\in\mathcal F_{2k}}C(W)$. The fact that $\varrho$ is a
winning strategy for Two in the Menger game on $\mathbb R$ guarantees
that $X$ will be covered by Two through this procedure.

\begin{example}
\label{luzin-michael}

If there is a Luzin subset of the real line,\footnote{That is, an
  uncountable set $L\subseteq\mathbb R$ such that $L\cap A$ is
  countable for every nowhere dense subset $A$ of $\mathbb R$. The
  Continuum Hypothesis implies the existence of a Luzin set
  \cite[Theorem 1]{luzin}.} then there is a Hausdorff non-regular
space $X$ that is Lindel\"of scattered and such that $X_\delta$ is not
Lindel\"of.

\end{example}

Let $L\subseteq\mathbb R$ be a Luzin set, which we may assume to
consist only of irrational numbers.
On the set $X=L\cup\mathbb Q$, consider the topology in which every
point of $L$ is isolated and basic neighbourhoods of $q\in\mathbb Q$
are of the form $\{q\}\cup\{x\in L:|x-q|<\frac{1}{n+1}\}$ for
$n\in\omega$. It is clear that $X$ is scattered and that $X_\delta$,
being discrete and uncountable, is not Lindel\"of. Yet $X$ is
Lindel\"of: from any open cover $\mathcal U$ of $X$, we can extract a
countable subset $\mathcal U_0$ that covers $\mathbb Q$; as $L$ is a
Luzin set, $\mathcal U_0$ leaves only countably many points of $L$
uncovered.

\begin{defin}
\label{def.R-cover}

An open cover $\mathcal U$ of a topological space $X$ is an
\emph{$R$-cover} if every Rothberger subspace of $X$ is
included in some element of $\mathcal U$. The set of all $R$-covers of
$X$ will be denoted by $\mathcal R_X$.

\end{defin}

The following result is a straightforward generalization of the
implication $(3)\rightarrow(1)$ of Theorem 17 of \cite{sch1}; its
proof is essentially the same.

\begin{prop}
\label{S1(R,O)}

A topological space $X$ satisfies $\mathsf S_1(\mathcal R,\mathcal O)$
if and only if $X$ is a Rothberger space.

\end{prop}

\begin{corol}
\label{C->R}

Let $X$ be a topological space such that every compact subspace of $X$
is Rothberger. Then $X$ satisfies $\mathsf S_1(\mathcal K,\mathcal O)$
if and only if $X$ is a Rothberger space.\footnote{This has also been
observed independently in \cite{bab}.}

\end{corol}

\begin{proof}

Since in this case we have $\mathcal R_X\subseteq\mathcal K_X$, it
follows that $\mathsf S_1(\mathcal K_X,\mathcal O_X)$ implies $\mathsf
S_1(\mathcal R_X,\mathcal O_X)$ --- which, by Proposition
\ref{S1(R,O)}, is equivalent to $X$ being Rothberger.
\end{proof}

\begin{corol}
\label{C->scat}

Let $X$ be a topological space every compact subspace of which
has an isolated point. Then $X$ satisfies $\mathsf S_1(\mathcal
K,\mathcal O)$ if and only if $X$ is a Rothberger space.

\end{corol}

\begin{proof}

This follows directly from Corollary \ref{C->R}, since every compact
scattered space is Rothberger.
\end{proof}

\begin{example}
\label{cov<d}

If $\mathrm{cov}(\mathcal M)<\mathfrak d$, then there is a Menger
regular space that does not satisfy $\mathsf S_1(\mathcal K,\mathcal
O)$.

\end{example}

It follows from Theorem 5 of \cite{fm} that $\mathrm{cov}(\mathcal
M)$ is the least cardinality of a non-Rothberger subspace of the real
line; let then $X\subseteq\mathbb R$ be such a subspace. As
$|X|<\mathfrak d$, it follows from Theorem 5 of \cite{folgen} (see
also \cite[Theorem 3]{fm}) that $X$ is a Menger space. By the
\v Cech-Posp\'i\v sil Theorem (\cite{cechp}; see e.g.
\cite[Theorem 7.19]{hodel}), any compact subspace of $X$ without
isolated points would have size at least $\mathfrak c\ge\mathfrak
d>|X|$, which is impossible; thus, by Corollary \ref{C->scat}, $X$
does not satisfy $\mathsf S_1(\mathcal K,\mathcal O)$.

\begin{example}
\label{sierp+sorg}

If there is a Sierpi\'nski subset of the real line,\footnote{That is,
  an uncountable set $S\subseteq\mathbb R$ such that $S\cap A$ is
  countable for every (Lebesgue) measure zero subset $A$ of $\mathbb
  R$. The Continuum Hypothesis implies the existence of a Sierpi\'nski
  set \cite[4.6]{sierp24}.}
then there is a Menger regular space that does not satisfy $\mathsf
S_1(\mathcal K,\mathcal O)$.

\end{example}

Pick a Sierpi\'nski set $S\subseteq\mathbb R$ and endow it with the
Sorgenfrey line topology. By Corollary 3.6 of \cite{sakai}, $S$ is
Menger. Since $S$ does not have measure zero, it cannot be Rothberger
\cite{roth}; thus, as every compact subset of the Sorgenfrey line
is countable, it follows from Corollary \ref{C->R} that $S$ does not
satisfy $\mathsf S_1(\mathcal K,\mathcal O)$.\\

In view of Examples \ref{cov<d} and \ref{sierp+sorg}, it is natural to
ask:

\begin{prob}
\label{zfc.dual?}

Is it consistent with $\mathrm{ZFC}$ that every Menger regular space
satisfies $\mathsf S_1(\mathcal K,\mathcal O)$?

\end{prob}

We conjecture that the answer is negative. Note that this could be
proven by means of a dichotomic argument, e.g. by showing that a
counterexample exists under $\mathrm{cov}(\mathcal M)=\mathfrak d$.
Should this be the case, one might still ask:

\begin{prob}
\label{zfc?}

Is there a $\mathrm{ZFC}$ example of a Menger regular space that does
not satisfy $\mathsf S_1(\mathcal K,\mathcal O)$?

\end{prob}

We point out that a set of reals satisfying these conditions
would have to be, in particular, a $\mathrm{ZFC}$ example
of a non-$\sigma$-compact Menger subspace of $\mathbb R$
--- a kind of set that only recently has been constructed
(see Theorem 16 in \cite{bart-tsaban}).
Note also that, if the regularity requirement is dropped, then Example
\ref{R-ctbl} answers Problem \ref{zfc?} in the affirmative.

%

\end{document}